\documentclass[10pt]{amsart}

\usepackage{amssymb,latexsym, mathtools,tikz}
\usepackage{enumerate}
\usepackage{graphicx}
\usepackage{float}
\usepackage{placeins}
\usepackage{mdframed}
\usepackage{amssymb}
\usepackage{esint}
\usepackage{cool}
\usepackage[all,cmtip]{xy}
\usepackage{mathtools}
\usepackage{amstext} 
\usepackage{array}   
\usepackage[shortlabels]{enumitem}
\usepackage{ytableau}
\usepackage{tikz}
\usetikzlibrary{arrows,decorations.markings, cd}
\tikzstyle arrowstyle=[scale=1]
\tikzstyle directed=[postaction={decorate,
decoration={markings,mark=at position .65 with {\arrow[arrowstyle]{stealth}}}}]
\usepackage{mathdots}
\usepackage{quiver}

\newcolumntype{L}{>{$}l<{$}} 
\newcolumntype{C}{>{$}c<{$}}
\newtheorem{theorem}{Theorem}[section]
\newtheorem{lemma}[theorem]{Lemma}
\newtheorem{cor}[theorem]{Corollary}
\newtheorem{prop}[theorem]{Proposition}
\newtheorem{setup}[theorem]{Setup}
\theoremstyle{definition}
\newtheorem{definition}[theorem]{Definition}

\newtheorem{obs}[theorem]{Observation}
\newtheorem{notation}[theorem]{Notation}

\theoremstyle{remark}
\newtheorem{remark}[theorem]{Remark}

\newtheorem{the context}[theorem]{The Context}
\newtheorem{question}[theorem]{Question}
\numberwithin{equation}{theorem}
\numberwithin{equation}{section}

\newcommand{\cat}[1]{\mathcal{#1}}

\newcommand{\rank}{\operatorname{rank}}

\newcommand{\grade}{\operatorname{grade}}

\newcommand{\coker}{\operatorname{Coker}}

\newcommand{\im}{\operatorname{Im}}

\newcommand{\Ker}{\operatorname{Ker}}

\newcommand{\ideal}[1]{\mathfrak{#1}}
\newcommand{\m}{\ideal{m}}

\renewcommand{\geq}{\geqslant}
\renewcommand{\leq}{\leqslant}
\renewcommand{\ker}{\Ker}
\renewcommand{\hom}{\Hom}

\newcommand{\Hom}{\operatorname{Hom}}

\newcommand{\maps}[5]{\xymatrix{#1 \ar[r]^-{#3} & #2 \\
#4 \ar@{|->}[r] & #5 \\}}

\newcommand{\mfa}{\mathfrak{a}}

\def\w{\wedge}

\def\im{\operatorname{im}}

\newcommand{\fitt}{\operatorname{Fitt}}

\newcommand{\sgn}{\operatorname{sgn}}

\newcommand{\tm}{\operatorname{tm}}

\begin{document}
\title{Products of Ideals and Golod Rings}



\author{Keller VandeBogert }
\date{\today}

\maketitle

\begin{abstract}
    In this paper, we study conditions guaranteeing that a product of ideals defines a Golod ring. We show that for a $3$-dimensional regular local ring (or $3$-variable polynomial ring) $(R , \m)$, the ideal $I \m$ always defines a Golod ring for any proper ideal $I \subset R$. We also show that non-Golod products of ideals are ubiquitous; more precisely, we prove that for any proper ideal with grade $\geq 4$, there exists an ideal $J \subseteq I$ such that $IJ$ is not Golod. We conclude by showing that if $I$ is any proper ideal in a $3$-dimensional regular local ring and $\mfa \subseteq I$ a complete intersection, then $\mfa I$ is Golod. 
\end{abstract}

\section{Introduction}

Let $(R , \m ,k )$ denote a local ring with embedding dimension $e$ and codepth $d$. Serre established a coefficient-wise inequality for the Poincar\'e series of $R$
$$P^R_k (t) \leq \frac{(1+t)^e}{1 - \sum_{j=1}^d \rank_k H_j (K^R) t^{j+1}},$$
where $K^R$ denotes the Koszul complex on a minimal generating set of $\m$. Later, Golod studied rings for which the above inequality is an equality, linking this ``maximality" condition to the existence of certain higher order homology operations, known as trivial Massey operations (see \cite{golod1962}). Today, rings for which Serre's inequality is equality are known as Golod rings, and uncovering conditions that imply Golodness has been a widely studied problem for multiple decades now; see for instance \cite{peeva19960}, \cite{srinivasan1989algebra}, \cite{herzog2013ordinary}, \cite{christensen2018golod}, or \cite{herzog2018koszul} (and the references therein).

The existence of a trivial Massey operation implies that the Koszul homology algebra of a Golod ring is trivial; that is, there are no nontrivial products between elements of positive degree. In \cite[Theorem 5.1]{berglund2007golod}, it was claimed that for monomial ideals, trivial multiplication on the Koszul homology algebra is in fact equivalent to Golodness. It turned out that there was a gap in the proof, with Katth\"an later producing a non-Golod ring (defined by a monomial ideal) with trivial product on the Koszul homology (see \cite{katthan2017non}). Likewise, it was a question of Welker as to whether all products of proper homogeneous ideals define Golod rings; this question was answered in the negative by De Stefani in \cite{de2016products} with a rather simple counterexample: the product of the ideals $(x_1 , \dots , x_4)$ and $(x_1^2 , \dots , x_4^2) \subset k[x_1 , \dots , x_4]$ does not define a Golod ring. It turns out that in the $3$ variable case, products of monomial ideals \emph{do} define Golod rings \cite{dao2020monomial}; the case for arbitrary ideals in a $3$-variable polynomial ring remains open.

The above examples serve to illustrate that Golod ideals remain mysterious even in low dimensional rings, and proving that a particular class of ideals define Golod rings is often a difficult problem. The first case where this problem becomes nontrivial is for rings of codepth $3$, since in codepth $1$, all rings are Golod, and in codepth $2$ the ring is either Golod or a complete intersection. In this paper, we attempt to introduce new techniques for studying rings defined by products of ideals. We prove that for an arbitrary $3$-dimensional regular local ring (or homogeneous ideal in a $3$-variable polynomial ring) $(R,\m)$, the product $I \m$ defines a Golod ring for any proper ideal $I$ (see Theorem \ref{thm:ImGolod}). We also show that Golodness of a product of ideals can be checked by verifying that a certain morphism of complexes is a split injection (see Proposition \ref{prop:golodSplit}). This allows us to prove that if $I$ is a complete intersection and $J \subset I$ is generated by some subset of the minimal generators of $I$, then $IJ$ defines a Golod ring.

We also prove that non-Golod products of ideals are in fact quite ubiquitous. More precisely, let $R$ be any regular local ring or standard graded polynomial ring over a field $k$. Then we prove that for any proper ideal $I$ and complete intersection $\mfa \subset \m I$ with $\mu ( \mfa ) \geq 4$, the ideal $\mfa I$ does \emph{not} define a Golod ring (see Proposition \ref{prop:nonGolodRing}). This generalizes the example given by De Stefani in \cite{de2016products}. In particular, for any ideal $I$ with grade $\geq 4$, one can find an ideal $J \subset I$ such that $IJ$ does not define a Golod ring. In the case that $\grade (I) \leq 3$, we prove the opposite; that is, $\mfa I$ is a Golod ring for any complete intersection $\mfa \subseteq I$.

The techniques in this paper involve understanding free resolutions of products of ideals and the lifts of certain types of basis elements to the Koszul homology algebra. Suppose $I = (a_1 , \dots , a_n)$ and $J$ are proper ideals; one way to construct a free resolution of $IJ$ is to construct the resolution of $(a_1 , \dots , a_k) + J (a_{k+1} , \dots , a_n)$ in an iterative fashion. As it turns out, this is handled by the machinery of trimming complexes as in \cite{vandebogert2020trimming}; moreover, it was recently shown how to construct a DG-algebra structure on trimming complexes of length $3$ in \cite{vandebogert2020dg}. This yields information about the Tor-algebra structure (and hence Golodness) for products of ideals. 

The paper is organized as follows. In Section \ref{sec:background}, we give a quick rundown of DG-algebras, trivial Massey operations, and Golod rings. We also provide a concise introduction to iterated trimming complexes along with a statement about the algebra structure in the length $3$ case. We conclude with an observation on the contribution of the $\mfa_i$ ideals of Setup \ref{set:trimmingcxSetup} to the Koszul homology of the trimmed ideal. In Section \ref{sec:ofFormIm}, we prove a first main result of the paper; namely, that for any $3$-dimensional regular local ring, the ideal $I \m$ defines a Golod ring. The proof actually follows quite easily by results previously established in \cite{vandebogert2020dg}.

In Section \ref{sec:nonGolodProds}, we study Golod and non-Golod rings defined by products of ideals in higher dimensional rings. As a first result, we prove that if certain maps appearing in the construction of iterated trimming complexes are split injections, then the associated product of ideals must define a Golod ring. Next, we consider an ideal $I$ containing any complete intersection $\mfa \subset I$. We first construct an explicit free resolution of $R / \mfa I$ that is minimal in the case $\mfa \subseteq \m I$, and then show that this resolution has nontrivial multiplication in the Tor-algebra if $\mu (\mfa ) \geq 4$. This in particular proves that $\mfa I$ cannot be a Golod ring. We then prove that if, in the above situation, $I$ has grade $\leq 3$ and $\mfa \subset I$, then $\mfa I$ \emph{does} define a Golod ring. We conclude with discussion and further questions.

\section{Golod Rings and Trimming Complexes}\label{sec:background}

In this section, we lay the groundwork for the rest of the paper. The first order of business involves defining trivial Massey operations and then Golod rings. We define Golod rings in terms of the existence of trivial Massey operations since this will be our primary means of proving Golodness in later sections. Next, we recall the construction of iterated trimming complexes as in \cite{vandebogert2020trimming} and the algebra structure in the length $3$ case. The following notation will be in play for the rest of the paper:

\begin{notation}
The notation $(F_\bullet , d_\bullet)$ will denote a complex $F_\bullet$ with differentials $d_\bullet$. When no confusion may occur, $F$ may be written, where the notation $d^F$ is understood to mean the differential of $F$ (in the appropriate homological degree).

Given a complex $F_\bullet$ as above, elements of $F_n$ will often be denoted $f_n$, without specifying that $f_n \in F_n$.
\end{notation}

\begin{definition}\label{def:dga}
A \emph{differential graded algebra} $(F,d)$ (DG-algebra) over a commutative Noetherian ring $R$ is a complex of free $R$-modules with differential $d$ and with a unitary, associative multiplication $F \otimes_R F \to F$ satisfying
\begin{enumerate}[(a)]
    \item $F_i F_j \subseteq F_{i+j}$,
    \item $d_{i+j} (x_i x_j) = d_i (x_i) x_j + (-1)^i x_i d_j (x_j)$,
    \item $x_i x_j = (-1)^{ij} x_j x_i$, and
    \item $x_i^2 = 0$ if $i$ is odd,
\end{enumerate}
where $x_k \in F_k$.
\end{definition}

The next definition will be essential for defining Golod rings. If $z_\lambda$ is a cycle in some complex $A$, then the notation $[z_\lambda]$ denotes the homology class of $z_\lambda$. In the following definition, $\overline{a} := (-1)^{|a|+1} a$, where $|a|$ denotes the homological degree of $a \in A$.  

\begin{definition}
Let $A$ be a DG-algebra with $H_0 (A) \cong k$. Then $A$ admits a \emph{trivial Massey operation} if for some $k$-basis $\mathcal{B} = \{ h_\lambda \}_{\lambda \in \Lambda}$, there exists a function
$$\mu : \coprod_{i=1}^\infty \cat{B}^i \to A$$
such that
\begingroup\allowdisplaybreaks
\begin{align*}
    &\mu ( h_\lambda) = z_\lambda \quad \textrm{with} \quad [z_\lambda] = h_\lambda, \ \textrm{and} \\
    &d \mu (h_{\lambda_1} , \dots , h_{\lambda_p} ) = \sum_{j=1}^{p-1} \overline{\mu (h_{\lambda_1} , \dots , h_{\lambda_j})} \mu (h_{\lambda_{j+1}} , \dots , h_{\lambda_p}). 
\end{align*}
\endgroup
\end{definition}

Observe that taking $p=2$ in the above definition yields that $H_{\geq 1} (A)^2 = 0$, so the induced algebra structure on $H(A)$ is trivial for a DG-algebra admitting a trivial Massey operation.

\begin{definition}\label{def:Golod}
Let $(R,\m)$ be a local ring and let $K^R$ denote the Koszul complex a minimal set of generators of $\m$. If $K^R$ admits a trivial Massey operation $\mu$, then $R$ is called a \emph{Golod ring}. An ideal $I$ in some ring $Q$ will be called \emph{Golod} if the quotient $Q/I$ is Golod.
\end{definition}

\begin{setup}\label{set:trimmingcxSetup}
Let $(R , \m , k)$ be a regular local ring. Let $I \subseteq R$ be an ideal and $(F_\bullet, d_\bullet)$ a free resolution of $R/I$. 

Write $F_1 = F_1' \oplus \Big( \bigoplus_{i=1}^t Re_0^i \Big)$, where, for each $i=1, \dotsc , t$, $e^i_0$ generates a free direct summand of $F_1$. Using the isomorphism
$$\hom_R (F_2 , F_1 ) = \hom_R (F_2,F_1') \oplus \Big( \bigoplus_{i=1}^t \hom_R (F_2 , Re^i_0) \Big)$$
write $d_2 = d_2' + d_0^1 + \cdots + d^t_0$, where $d_2' \in \hom_R (F_2,F_1')$ and $d^i_0 \in \hom_R (F_2 , Re^i_0)$.  For each $i=1, \dotsc , t$, let $\mfa_i$ denote any ideal with
$$d^i_0 (F_2) \subseteq \mfa_i e^i_0,$$
and $(G^i_\bullet , m^i_\bullet)$ be a free resolution of $R/\mfa_i$. Use the notation $K' := \im (d_1|_{F_1'} : F_1' \to R)$, $K^i_0 := \im (d_1|_{Re^i_0} : Re^i_0 \to R)$, and let $J := K' + \mfa_1 \cdot K^1_0+ \cdots + \mfa_t \cdot K_0^t$.
\end{setup}

\begin{remark}
In Setup \ref{set:trimmingcxSetup}, one may alternatively assume that $R$ is a standard graded polynomial ring over a field $k$ and all over input data is homogeneous.
\end{remark}

\begin{theorem}[\cite{vandebogert2020trimming}]\label{itres}
Adopt notation and hypotheses as in Setup \ref{set:trimmingcxSetup}. Then there exists a morphism of complexes of the following form:
\begin{equation}\label{itcomx}
\xymatrix{\cdots \ar[r]^{d_{k+1}} &  F_{k} \ar[ddd]^{\begin{pmatrix} q_{k-1}^1 \\
\vdots \\
q_{k-1}^t \\
\end{pmatrix}}\ar[r]^{d_{k}} & \cdots \ar[r]^{d_3} & F_2 \ar[rrrr]^{d_2'} \ar[ddd]^{\begin{pmatrix} q_1^1 \\
\vdots \\
q_1^t \\
\end{pmatrix}} &&&& F_1' \ar[ddd]^{d_1'} \\
&&&&&&& \\
&&&&&&& \\
\cdots \ar[r]^-{\bigoplus m^i_k} & \bigoplus_{i=1}^t G^i_{k-1} \ar[r]^-{\bigoplus m^i_{k-1}} & \cdots \ar[r]^-{\bigoplus m^i_2} & \bigoplus_{i=1}^t G^i_1 \ar[rrrr]^-{-\sum_{i=1}^t  d_1(e^i_0)m_1^i} &&&& R \\}\end{equation}
where $d_1' : F_1 \to R$ is the restriction of $d_1$ to $F_1$. Moreover, the mapping cone of \ref{itcomx} is a free resolution of $R/J$. The mapping cone of \ref{itcomx} will be denoted $T_\bullet$ with differentials $\ell_\bullet$.
\end{theorem}

\begin{definition}\label{def:ittrimcx}
The \emph{iterated trimming complex} associated to the data of Setup \ref{set:trimmingcxSetup} is the complex of Theorem \ref{itres}.
\end{definition}

The following theorem is not stated in its entirety as in \cite{vandebogert2020dg}; this is because writing out all of the products is quite technical and not totally enlightening. In our situation, the more closed forms of the products given by Proposition $3.4$ of \cite{vandebogert2020dg} will be used, since the DG-module hypotheses will be trivially verified for our purposes. The reader is encouraged to consult \cite{vandebogert2020dg} for the more grisly details.

\begin{theorem}\label{thm:DGlength3}
Adopt notation and hypotheses as in Setup \ref{set:trimmingcxSetup}, and assume that the complexes $F_\bullet$ and $G_\bullet^i$ ($1 \leq i \leq t$) are length $3$ DG-algebras. Then the length 3 iterated trimming complex $T_\bullet$ of Theorem \ref{itres} admits the structure of an associative DG-algebra, and the product on $T_\bullet$ can be written down in terms of the products on $F_\bullet$ and $G_\bullet$.
\end{theorem}

\begin{cor}[{\cite[Corollary 3.6]{vandebogert2020dg}}]\label{cor:nontrivialMults}
Adopt notation and hypotheses as in Setup \ref{set:trimmingcxSetup}. Assume that the complexes $F_\bullet$ and $G_\bullet^i$ (for $1 \leq i \leq t$) are minimal. Then the only possible nontrivial products in the algebra $\overline{T_\bullet}$ are
$$\overline{F_1'} \cdot_T \overline{F_1'}, \quad \overline{F_1'} \cdot_T \overline{F_2}, \quad \overline{F_1'} \cdot_T \overline{G_1^i} \quad \overline{G_1^i} \cdot_T \overline{F_2}, \quad \textrm{and} \quad \overline{F_1'} \cdot_T \overline{G_2^i} \quad (1 \leq i \leq t).$$
\end{cor}

Let $A_\bullet$ and $B_\bullet$ be complexes with $H_0 (A) = R$ and $H_0 (B) = S$. Recall that there is a functorial isomorphism
$$H_\bullet (A \otimes S) \cong H_\bullet (R \otimes B)$$
induced by the natural projections
\[\begin{tikzcd}
	A & {A \otimes B} & B
	\arrow[from=1-2, to=1-1]
	\arrow[from=1-2, to=1-3]
\end{tikzcd}\]
The isomorphism can be explicitly described as follows:
\begin{enumerate}
    \item Choose a cycle $z_1$ in $A \otimes S$ representing a basis element of $H_\bullet (A \otimes S)$.
    \item Lift $z_1$ to a cycle $z_2$ in $A \otimes B$.
    \item Project $z_2$ onto a cycle $z_3$ in $R \otimes B$, then descend to homology.
\end{enumerate}

The following observation shows that the contribution of the ideals $\mfa_i$ to the Koszul homology of the trimmed ideal is completely determined by the Koszul homology $H_\bullet (R / \mfa_i)$, for each $1 \leq i \leq t$.

\begin{obs}\label{obs:KoszulLifts}
Adopt notation and hypotheses as in Setup \ref{set:trimmingcxSetup}, and assume that the complexes $F_\bullet$ and $G_\bullet^i$ ($1 \leq i \leq t$) are minimal. Let $\phi^i : G^i_\bullet \to R / \mfa_i \otimes K_\bullet$ denote any map lifting the isomorphism $G^i_\bullet \otimes k \cong H_\bullet (R/ \mfa_i \otimes K_\bullet)$. 

Then, the isomorphism $H_\bullet (T_\bullet \otimes k) \cong H_\bullet (R/J \otimes K_\bullet)$ sends $\overline{g_j^i} \in G^i_j \otimes k$ to the element $[\phi^i (g_j^i) d_1 (e_0^i)] \in H_j (R/J \otimes K_\bullet)$.
\end{obs}

\begin{proof}
Observe that the only difference between $G_\bullet^i$ on its own and $G_\bullet^i$ viewed as a direct summand in the trimming complex $T_\bullet$ is that the first differential is rescaled by the element $d_1 (e_0^i)$ in the latter case. By definition of the isomorphism $H_\bullet (T_\bullet \otimes k) \cong H_\bullet (R/J \otimes K_\bullet)$, the result is immediate.
\end{proof}

\section{Products of the Form $I \m$ define Golod Rings in $3$-dimensional regular local rings}\label{sec:ofFormIm}

In this section, we prove that any product of the form $I \m$ defines a Golod ring if $I$ is a proper ideal in some $3$-dimensional regular local ring. We first use a simple observation about the generators of the Koszul homology to see that the possible Tor-algebra classes of a general product of ideals is rather restricted. We then recall some additional results proved in \cite{vandebogert2020dg} that will be essential for proving the main result. Observe that
\begin{center}
    $(*) \quad$ one can replace the regular local ring $R$ with a standard graded polynomial ring over a field $k$ for all results in this paper.
\end{center}
To begin with, we prove a general result about the Koszul homology algebra for products.

\begin{obs}\label{obs:trivialDeg1}
Let $R$ be a regular local ring and let $I$ and $J$ be two proper ideals. Then $H_1 (R/IJ) \cdot H_1 (R/IJ) = 0$.
\end{obs}

\begin{proof}
It is clear that $H_1 (R/IJ)$ is generated by the classes of elements of the form $d(z_1^I) z_1^J$, where $z_1^I$ and $z_1^J$ represent basis elements for $H_1 (R/I)$ and $H_1 (R/J)$, respectively. Now take the product and descend to homology. 
\end{proof}

Observe that by, for instance, \cite[Proposition 5.2.4]{avramov1998infinite}, it suffices to show that the multiplication in the Tor-algebra of $R/IJ$ is trivial to prove Golodness (in the case that $R / IJ$ has projective dimension $3$). This will end up being the method we use to prove Golodness for ideals of the form $I \m$.

\begin{cor}\label{cor:torAlgProd}
Let $R$ be a $3$-dimensional regular local ring and $I$, $J \subset R$ proper ideals. If $R/IJ$ has projective dimension $3$, then $R/IJ$ has Tor algebra class $G(r)$ or $H (0,q)$ for some $r \geq 2$ or $q \geq 0$.
\end{cor}

\begin{proof}
The statement about the Tor-algebras is clear by the Tor-algebra classification of \cite[Theorem 2.1]{avramov1988poincare} combined with Observation \ref{obs:trivialDeg1}.
\end{proof}

\begin{remark}
If there does exist a product of ideals defining a quotient of Tor-algebra class $G(r)$, then this quotient is necessarily non-Gorenstein by a result of Huneke \cite{huneke2007ideals}. This means that such a quotient would be another counterexample to a question of Avramov (see \cite{avramov2012cohomological}).
\end{remark}

Next, we recall some notation and results from \cite{vandebogert2020dg}.

\begin{notation}\label{not:trimmingNotation}
Let $I = (\phi_1 , \dots  , \phi_n) \subseteq R$ be an $\m$-primary ideal and $F_\bullet$ a DG-algebra free resolution of $R/I$. Given an indexing set $\sigma = \{1 \leq \sigma_1 < \cdots < \sigma_t \leq n\}$, define
$$\tm_\sigma (I) := (\phi_i \mid i \notin \sigma) + \m (\phi_j \mid j \in \sigma).$$
The transformation $I \mapsto \tm_\sigma (I)$ will be referred as \emph{trimming} the ideal $I$.
\end{notation}

Observe that if $\sigma = \{ 1 , \dots , n \}$ (in the notation of Notation \ref{not:trimmingNotation}), then $\tm_\sigma (I) = I \m$. Thus a free resolution of $R / I \m$ may be obtained as a trimming complex and the Tor algebra structure may be deduced from the product furnished by Theorem \ref{thm:DGlength3}.

\begin{setup}\label{set:trimmingSetup}
Let $(R , \m , k)$ denote a regular local ring of dimension $3$. Let $I = (\phi_1 , \dotsc , \phi_n) \subseteq R$ be an ideal and $\sigma = (1 \leq \sigma_1 < \cdots < \sigma_t \leq n)$ be an indexing set. Let $(F_\bullet, d_\bullet)$ and $(K_\bullet, m_\bullet)$ be minimal DG-algebra free resolutions and $R/I$ and $k$, respectively. By Theorem \ref{itres}, a free resolution of $R/ \tm_\sigma (I)$ may be obtained as the mapping cone of a morphism of complexes of the form:
\begin{equation}\label{eq:trimcx}
    \xymatrix{0  \ar[rr] && F_3 \ar[d]^-{Q_2} \ar[rr]^-{d_3} && F_2 \ar[d]^-{Q_1} \ar[rrr]^-{d_2'} &&& F_1' \ar[d]^-{d_1} \\
\bigoplus_{i=1}^t K_3 \ar[rr]^-{\bigoplus_{i=1}^t m_3} && \bigoplus_{i=1}^t K_2 \ar[rr]^-{\bigoplus_{i=1}^t m_2} && \bigoplus_{i=1}^t K_1 \ar[rrr]^-{-\sum_{i=1}^t m_1(-) d_1(e_0^{\sigma_i})} &&& R, \\}
\end{equation}
where
$$F_1' := \bigoplus_{j \notin \sigma} Re_0^j \quad \textrm{and} \quad d_2' : F_2 \xrightarrow{d_2} F_1 \xrightarrow{\textrm{proj}} F_1'.$$
Let $T_\bullet$ denote the mapping cone of \ref{eq:trimcx}. 
\end{setup}

In the following, note that $\overline{\cdot} := \cdot \otimes k$, where $k$ denotes the residue field.

\begin{prop}[{\cite[Lemma 4.9]{vandebogert2020dg}}]\label{prop:boundedProds}
Adopt notation and hypotheses as in Setup \ref{set:trimmingSetup} and assume that $I \subseteq \m^2$. Then,
$$\overline{F_1} \cdot_F \overline{F_1} \subseteq \ker (Q_1 \otimes k)$$
$$\overline{F_1} \cdot_F \overline{F_2} \subseteq \ker (Q_2 \otimes k)$$
In particular, the only possible nontrivial products in the algebra $\overline{T_\bullet}$ are given by
$$\overline{F_1'} \cdot_T \overline{F_1'} \quad \textrm{and} \quad \overline{F_1'} \cdot_T \overline{F_2}.$$
\end{prop}

Combining the above results, we are able to prove the following result. Notice that the assumption that $R$ is a regular local ring here is essential, as the result \cite[Theorem 4.2]{christensen2018golod} shows that even over a complete intersection ring, powers of the maximal ideal may not define Golod rings.

\begin{theorem}\label{thm:ImGolod}
Let $(R, \m ,k)$ be a $3$-dimensional regular local ring and $I$ a proper ideal of $R$. Then $I \m$ defines a Golod ring.
\end{theorem}

\begin{proof}
Assume first that $I \subseteq \m^2$. By Proposition \ref{prop:boundedProds}, the only possible nontrivial products come from products of the form $\overline{F_1'} \cdot_T \overline{F_1'}$ or $\overline{F_1'} \cdot_T \overline{F_2}$. However, $F_1' = 0$, so all products are trivial.

If $I \subseteq \m$ but $I \not\subset \m^2$, then $I$ is either a complete intersection or a hyperplane section that is \emph{not} a complete intersection; the latter case follows from Lemma $4.13$ of \cite{vandebogert2020dg} and the case that $I$ is a complete intersection will be handled by Proposition \ref{prop:golodProdforleq3}.
\end{proof}

\section{On the Ubiquity of Non-Golod Products of Ideals}\label{sec:nonGolodProds}

In this section, we prove more results on Golodness and non-Golodness of certain types of quotients. We show that Golodness of a product of ideals can be detected by showing that an associated morphism of complexes is a split injection; this allows us to prove that certain products of complete intersections define Golod rings. We then consider a method of constructing non-Golod products of ideal en masse. To do this, we show how to construct a free resolution for particular products of ideals. Then, we show that under sufficient assumptions this complex is guaranteed to have nontrivial multiplication in the Tor-algebra; this proves that such ideals cannot be Golod. We then prove that if the ideals as above have sufficiently small grade, then the product \emph{is} Golod. 

The first lemma of this section is a general result on constructing trivial Massey operations. It will be used to give a quick proof of Proposition \ref{prop:prodKoszul}. Throughout this section, $(R , \m)$ will denote a regular local ring (or a standard graded polynomial ring over a field). In the following statement, let $K_\bullet$ denote the Koszul complex on the minimal generators of $\m$.

\begin{lemma}\label{lem:golodnessLem}
Let $I \subset R$ be any ideal and enumerate a $k$-basis $\cat{B} = \{ [z_a ] \}_{a \in A}$ for the Koszul homology algebra $H_{\geq 1} (R/I)$. Assume that there exists a map $\nu : \{ z_a \mid a \in A \} \to R/I \otimes K_\bullet$ such that
$$(*) \qquad z_a \w z_b = z_a \w d( \nu (z_b)) \quad \textrm{for all} \ a,  \ b \in A.$$
Then $R/I$ is a Golod ring.
\end{lemma}

\begin{proof}
Let $\mu : \coprod_{i=1}^\infty \cat{B}^i \to R/I \otimes K_\bullet$ be defined by
$$\mu ([z_{a_1}] , \dots , [z_{a_p}] ) := z_{a_1} \w \nu (z_{a_2}) \w \cdots \w \nu (z_{a_p}).$$
Then the proof follows by showing that $\mu$ as above is a trivial Massey operation on $R/I \otimes K_\bullet$. Inductively, one computes:
\begingroup\allowdisplaybreaks
\begin{align*}
    &\sum_{i=1}^{p-1} \overline{\mu ([z_{a_1}] , \dots , [z_{a_i}])} \mu ([z_{a_{i+1}}] , \dots , [z_{a_p}]) \\
    =& \sum_{i=1}^{p-1} \overline{z_{a_1} \w \nu (z_{a_2}) \w \cdots \w \nu (z_{a_i})} \w z_{a_{i+1}} \w \cdots \w \nu (z_{a_p}) \\
    =& \sum_{i=1}^{p-1} \overline{z_{a_1} \w \nu (z_{a_2}) \w \cdots \w \nu (z_{a_i})} \w d ( \nu (z_{a_{i+1}})) \w \cdots \w \nu (z_{a_p}) \\
    =& (-1)^{|z_{a_1}|} z_{a_1} \w d \big( \nu (z_{a_2}) \w \cdots \w \nu (z_{a_p}) \big) \\
    =& d \big( z_{a_1} \w \nu (z_{a_2}) \w \cdots \w \nu (z_{a_p}) \big) \\
    =& d \mu ( [z_{a_1}] , \dots , [z_{a_p}]).
\end{align*}
\endgroup
\end{proof}

\begin{notation}
Let $I$ and $J$ be proper ideals of $R$. Elements of the Koszul homology $H_i (R/I)$ and $H_j (R/J)$ will be denoted $z_i^I$ and $z_j^J$, respectively. The homology class of a cycle will in $H_\bullet (M)$ will be denoted by $[ \cdot]_{M \otimes K}$ or just $[ \cdot]$, when no confusion may occur.
\end{notation}

\begin{prop}\label{prop:prodKoszul}
Let $I$ and $J$ be proper ideals of $R$. If the Koszul homology $H_\bullet (R / IJ)$ is generated by the classes of elements of the form $d(z_i^I) \w z_j^J$, then $IJ$ defines a Golod ring.
\end{prop}

\begin{proof}
Let $\sum_{\ell=1}^m c_\ell [d(z_{a_\ell}^I) \w z_{b_\ell}^J]$ be a basis element for the Koszul homology, where $c_\ell \in k$ for each $\ell$; define:
$$\nu \Big( \sum_{\ell=1}^m c_\ell d(z_{a_\ell}^I) \w z_{b_\ell}^J \Big) := \sum_{\ell=1}^m c_\ell z_{a_\ell}^I \w z_{b_\ell}^J.$$
It is clear that the condition $(*)$ of Lemma \ref{lem:golodnessLem} is satisfied with respect to this basis, whence the result.
\end{proof}

In the statements that follow, we will often make the assumption that $\fitt_R (I) \subseteq J$ (where $\fitt_R (I)$ denotes the Fitting ideal of $I$; that is, the ideal generated by the entries of any minimal presentation matrix for $I$). This is done in order to satisfy the condition
$$d_0^i (F_2) \subseteq \mfa_i e_0^i$$
in the notation and hypotheses of Setup \ref{set:trimmingcxSetup}. If this is not satisfied, then the trimming procedure must be done iteratively, trimming each generator one at a time with a different ideal at each step. The following Proposition gives an interesting criterion for Golodness in terms of the vertical maps appearing in diagram \ref{itcomx} in the statement of Theorem \ref{itres}.

\begin{prop}\label{prop:golodSplit}
Let $I$ and $J$ be proper ideals of $R$ with $\fitt_R (I) \subseteq J$ and let $(F_\bullet, d_\bullet)$ and $(G_\bullet, m_\bullet)$ be minimal free resolutions of $R/I$ and $R/J$, respectively. Let $t = \mu (I)$ and assume that the vertical maps in the diagram 
$$\xymatrix{\cdots \ar[r]^{d_{k+1}} &  F_{k} \ar[dd]^{Q_{k-1}}\ar[r]^{d_{k}} & \cdots \ar[r]^{d_3} & F_2 \ar[dd]^{Q_1} &&&&  \\
&&&&&&& \\
\cdots \ar[r]^-{\bigoplus m_k} & \bigoplus_{i=1}^t G_{k-1} \ar[r]^-{\bigoplus m_{k-1}} & \cdots \ar[r]^-{\bigoplus m_2} & \bigoplus_{i=1}^t G_1 \ar[rrrr]^-{-\sum_{i=1}^t m_1(-)\cdot d_1(e^i_0)} &&&& R \\}$$
of Theorem \ref{itres} are split injections. Then $IJ$ defines a Golod ring. 
\end{prop}

\begin{proof}
Let $C_i := \coker Q_i$ where the $Q_i$ are as in the statement of the theorem. The minimal free resolution of $R/IJ$ is then obtained as the induced quotient complex
$$\cdots \to C_i \to C_{i-1} \to \cdots \to C_1 \to R.$$
Since each $Q_i$ is a split injection, each $C_i$ is a direct summand of $\bigoplus_{j=1}^t G_i$. Employing Observation \ref{obs:KoszulLifts}, it follows that the Koszul homology is generated by elements as in Proposition \ref{prop:prodKoszul}, whence $IJ$ defines a Golod ring.
\end{proof}

\begin{cor}
Let $J = (a_1 , \dots , a_n)$ be a regular sequence and $I = (a_1 , \dots , a_m)$ for some $m \leq n$. Then $IJ$ defines a Golod ring.
\end{cor}

\begin{proof}
Recall first that $\fitt_R (I) = I$ if $I$ is a complete intersection, so the assumption $\fitt_R (I) \subset J$ of Proposition \ref{prop:golodSplit} is satisfied. Let $K$ be a free $R$-module of rank $n$ with $\psi : K \to R$ a map of $R$-modules with $\psi (K) = J$. Then there exists a direct summand $K' \subset K$ such that $\psi (K') = I$, in which case the vertical maps of Proposition \ref{prop:golodSplit} may be chosen as the natural inclusion $\bigwedge^i K' \hookrightarrow \bigwedge^i K$, which is evidently a split injection. 
\end{proof}

In the following setup, recall that it is always possible to put a (possibly nonassociative) DG-algebra structure on an acyclic complex $F_\bullet$ with $F_0 = R$; this is proved in, for instance, \cite[Proposition 1.1]{buchsbaum1977algebra}. For our purposes, associativity will not be relevant since the algebra structure will only be used to construct a morphism of complexes.

\begin{setup}\label{set:NonGOlodSet}
Let $R$ be a regular local ring with $I \subset R$ any proper ideal and let $\mfa \subseteq  I $ be a complete intersection. Let $K_\bullet$ denote the Koszul complex resolving $R/\mfa$ and $F_\bullet$ the minimal free resolution of $R/I$. Let $\cdot_F$ denote any (not necessarily associative) DG-algebra product on $R/I$ and let $L_1 : K_1 \to F_1$ denote the comparison map in homological degree $1$ extending the identity map. 

For $i>1$, let $L_i : K_i \to F_i$ be the map sending $f_{\sigma} \mapsto L_1(f_{\sigma_1}) \cdot_F \cdots \cdot_F L_1 (f_{\sigma_i})$. Notice that since $\cdot_F$ is not necessarily associative, the above notation is understood to mean multiplication starting with the left-most terms first, then moving to the right\footnote{A more precise notation would say $\Big( \cdot \big(L_1(f_{\sigma_1}) \cdot_F L_1 (f_{\sigma_2}) \big) \cdot_F \cdots \cdot_F L_1(f_{\sigma_{i-1}}) \Big) \cdot_F L_1 (f_{\sigma_i})$}. With this, define $\Phi_i : K_i \to F_{i-1} \otimes K_1$ for $i >1$ to be the map sending
$$f_{\sigma} \mapsto \sum_{r \in \sigma} \sgn(r \in \sigma) L_{i-1} (f_{\sigma \backslash r}) \otimes f_r,$$
where $\sgn$ denotes the sign of the permutation that reorders sets into ascending order.
\end{setup}

\begin{prop}
Adopt notation and hypotheses as in Setup \ref{set:NonGOlodSet}. For all $i >2$, 
$$\Phi_{i-1} \circ d^K_i = d^F_{i-1} \circ \Phi_i$$
\end{prop}

\begin{proof}
One computes:
\begingroup\allowdisplaybreaks
\begin{align*}
    \Phi_{i-1} \circ d^K_i (f_\sigma) &= \Phi_{i-1} \Big( \sum_{r \in \sigma} \sgn (r \in \sigma) L_{i-1} (f_{\sigma \backslash r}) \otimes f_r \Big) \\
    &= \sum_{r \in \sigma} \sum_{s \in \sigma \backslash r} \sgn (r \in \sigma) \sgn (s \in \sigma \backslash r) \psi (f_s) L_{i-2} (f_{\sigma \backslash r,s} ) \otimes f_r \\
    &= d_{i-1}^F \circ \Phi_i (f_\sigma). 
\end{align*}
\endgroup
\end{proof}

The following theorem provides a free resolution of any quotient of the form $R / \mfa I$, where $\mfa I$ is obtained as in Setup \ref{set:NonGOlodSet}.

\begin{theorem}\label{thm:theMFR}
Adopt notation and hypotheses as in Setup \ref{set:NonGOlodSet}. Then the mapping cone of the morphism of complexes
$$\xymatrix{\cdots \ar[r]^{d_{k+1}^K} &  K_{k} \ar[dd]^{\Phi_{k-1}}\ar[r]^{d_{k}^K} & \cdots \ar[r]^{d_3^K} & K_2 \ar[dd]^{\Phi_1} &&&&  \\
&&&&&&& \\
\cdots \ar[r]^-{d^F_k \otimes 1} & F_{k-1} \otimes K_1 \ar[r]^-{d^F_{k-1} \otimes 1} & \cdots \ar[r]^-{d^F_2 \otimes 1} & F_1 \otimes K_1 \ar[rr]^-{-d^F_1 \otimes d_1^K} && R \\}$$
is a free resolution of $R/ \mfa I$. If $\mfa \subseteq \m I$, then this resolution is minimal. 
\end{theorem}

\begin{proof}
Simply observe that the diagram appearing in Theorem \ref{thm:theMFR} is a more concise way to write the trimming complex appearing in Theorem \ref{itres}. The $\Phi_i$ maps constructed in Setup \ref{set:NonGOlodSet} make the appropriate diagrams commute, so the mapping cone is a free resolution of $R / \mfa I$ as desired. If $\mfa \subseteq \m I$, then the maps $L_i$ as in Setup \ref{set:NonGOlodSet} satisfy $L_i (K_i) \subseteq \m F_{i-1} \otimes K_1$. This completes the proof.
\end{proof}

\begin{prop}\label{prop:nonGolodRing}
Adopt notation and hypotheses as in Setup \ref{set:NonGOlodSet} with $\mfa \subset \m I$. Assume $\mu ( \mfa ) \geq 4$. Then $\mfa I$ does not define a Golod ring. 
\end{prop}

In the following proof, we tacitly use the observation that the algebra structure provided by Theorem \ref{thm:DGlength3} also provides a partial (not necessarily associative) algebra structure on trimming complexes of longer length. Likewise, so as to not conflict with indexing notation/the residue field, let $K_1$ (as in Setup \ref{set:NonGOlodSet}) have basis $g_1 , \dots , g_n$ for some $1 \leq n \leq \dim R$. The notation $g_{ij}$ is shorthand for $g_i \w g_j \in K_2$.

\begin{proof}
The result will follow by showing that there is always nontrivial multiplication between certain types of elements of homological degree $2$. Let $T_\bullet$ denote the complex of Theorem \ref{thm:theMFR} and $\pi_K$ denote the projection of $T_\bullet$ onto $K_\bullet$. Using the product of Theorem \ref{thm:DGlength3} (with explicit form written in \cite[Theorem 3.3]{vandebogert2020dg}), one finds:
\begingroup\allowdisplaybreaks
\begin{align*}
    \pi_K \circ d^T ( g_{ij} \cdot_T g_{k\ell} ) &= - d^K (g_i) g_i \w g_{k \ell} + d^K (g_j) g_i \w g_{k \ell} \\
    &- d^K (g_k) g_{ij} \w g_\ell + d^K (g_\ell) g_{ij} \w g_k \\
    &= \pi_K \circ d^T (- g_{ij} \w g_{k \ell} ).
\end{align*}
\endgroup
This implies that $\cdot_T$ may be chosen (up to a cycle which disappears after tensoring with $k$) such that $\pi_K (g_{ij} \cdot_T g_{k \ell}) =  \pi_K ( -g_{ij} \w g_{k \ell})$; this of course implies that after tensoring with the residue field, $g_{ij} \cdot_T g_{k \ell}$ is nontrivial, whence $\mfa I$ cannot define a Golod ring. 
\end{proof}

In the $3$-dimensional case, we instead find that the opposite of Proposition \ref{prop:nonGolodRing} holds.

\begin{prop}\label{prop:golodProdforleq3}
Let $R$ be a $3$-dimensional regular local ring. Let $I$ be a proper ideal of $R$ and $\mfa \subseteq I$ any complete intersection. Then $\mfa I$ defines a Golod ring.
\end{prop}

\begin{proof}
Assume to avoid trivialities that $R / \mfa$ and $R / I$ both have projective dimension $3$ over $R$. A free resolution of $R/ \mfa I$ is obtained as in Theorem \ref{thm:theMFR}. Notice that by Corollary \ref{cor:nontrivialMults} there is only one possible product that could descend to a nontrivial product in the Tor-algebra, and that is given by the product $(F_1 \otimes K_1) \otimes K_2 \to (F_3 \otimes K_1) \oplus K_3$. One can compute a simple closed form for the product $(F_1 \otimes K_1) \otimes K_2 \to (F_3 \otimes K_1) \oplus K_3$ as
$$(f_1 \otimes e_0^i) \cdot_T g_{ij} := d^F (f_1) e_0^i \w g_{ij} + f_1 \cdot_F L_2( g_{ij} ) \otimes e_0^i.$$
It is straightforward to verify that this product is compatible with the products defined in Proposition $3.4$ of \cite{vandebogert2020dg}. The proof then follows after considering two cases:

\textbf{Case 1:} $I$ is a complete intersection. If $L_1 (K_1) \subseteq \m F_1$, then it is clear that the above product vanishes after tensoring with the residue field. Thus assume that $L_2 (g_{ij}) \notin \m F_2$; if this holds for all $1 \leq i , j \leq 3$, then $\mfa I$ is Golod by Proposition \ref{prop:golodSplit}. Otherwise, we may assume that $L_2 (g_{12} ) \notin \m F_2$ and $L_1 (g_{i3}) \in \m F_1$ for $i=1,2$; in this case, $g_{12}$ must vanish after descending to homology since $g_{13}$ and $g_{23} \in \ker \Phi_1 \otimes k$. Thus any product involving $g_{12}$ is trivial. Likewise, any product with $g_{i3}$ for $i=1,2$ is trivial since $L_2 (g_{i3}) \in \m F_1$ for $i=1,2$. In any of the above scenarios, the product is trivial after descending to homology.

\textbf{Case 2:} $I$ is not a complete intersection. This case follows immediately since $f_1 \cdot_F L_2( g_{ij}) \in \m F_3$ for the simple reason that the only class of ideals admitting a nontrivial triple product in the Tor-algebra are complete intersections by \cite[Theorem 2.1]{avramov1988poincare}.
\end{proof}

It is not clear how to extend the techniques related to trimming complexes to handle the case of arbitrary products of ideals. If the inclusion of Fitting ideals imposed above is not assumed, then the trimming procedure becomes considerably more complicated. In the graded case, it is possible that appropriate regularity bounds combined with careful degree counting will show that the algebra structure given in \cite{vandebogert2020dg} descends to a trivial algebra structure in homology. 

Instead of tackling the general case all at once, one could instead restrict to generic ideals in the $3$-variable polynomial ring. More precisely,

\begin{question}
Let $R = k[x,y,z]$ and $I$ and $J$ be two proper ideals of $R$ defining compressed rings. Does $IJ$ define a Golod ring?
\end{question}

Perhaps a proof of this statement would give insight into the general case (or intuition for building a counterexample).

\section*{Acknowledgements}

Thanks to Hailong Dao for suggesting a problem to me that resulted in this paper, and thanks to the anonymous referee for multiple improvements and corrections.

\bibliographystyle{amsplain}
\bibliography{biblio}

\providecommand{\bysame}{\leavevmode\hbox to3em{\hrulefill}\thinspace}
\providecommand{\MR}{\relax\ifhmode\unskip\space\fi MR }
\providecommand{\MRhref}[2]{%
  \href{http://www.ams.org/mathscinet-getitem?mr=#1}{#2}
}
\providecommand{\href}[2]{#2}
\begin{thebibliography}{10}

\bibitem{avramov1998infinite}
Luchezar~L Avramov, \emph{Infinite free resolutions}, Six lectures on
  commutative algebra, Springer, 1998, pp.~1--118.

\bibitem{avramov2012cohomological}
\bysame, \emph{A cohomological study of local rings of embedding codepth 3},
  Journal of Pure and Applied Algebra \textbf{216} (2012), no.~11, 2489--2506.

\bibitem{avramov1988poincare}
Luchezar~L Avramov, Andrew~R Kustin, and Matthew Miller, \emph{Poincar{\'e}
  series of modules over local rings of small embedding codepth or small
  linking number}, Journal of Algebra \textbf{118} (1988), no.~1, 162--204.

\bibitem{berglund2007golod}
Alexander Berglund and Michael J{\"o}llenbeck, \emph{On the {G}olod property of
  {S}tanley--{R}eisner rings}, Journal of Algebra \textbf{315} (2007), no.~1,
  249--273.

\bibitem{buchsbaum1977algebra}
David~A Buchsbaum and David Eisenbud, \emph{Algebra structures for finite free
  resolutions, and some structure theorems for ideals of codimension 3},
  American Journal of Mathematics \textbf{99} (1977), no.~3, 447--485.

\bibitem{christensen2018golod}
Lars~Winther Christensen and Oana Veliche, \emph{The {G}olod property of powers
  of the maximal ideal of a local ring}, Archiv der Mathematik \textbf{110}
  (2018), no.~6, 549--562.

\bibitem{dao2020monomial}
Hailong Dao and Alessandro De~Stefani, \emph{On monomial {G}olod ideals}, Acta
  Mathematica Vietnamica (2020), 1--9.

\bibitem{de2016products}
Alessandro De~Stefani, \emph{Products of ideals may not be {G}olod}, Journal of
  Pure and Applied Algebra \textbf{220} (2016), no.~6, 2289--2306.

\bibitem{golod1962}
E.~S. Golod, \emph{Homologies of some local rings}, Dokl. Akad. Nauk SSSR
  \textbf{144} (1962), 479--482. \MR{0138667}

\bibitem{herzog2013ordinary}
J{\"u}rgen Herzog and Craig Huneke, \emph{Ordinary and symbolic powers are
  {G}olod}, Advances in Mathematics \textbf{246} (2013), 89--99.

\bibitem{herzog2018koszul}
J{\"u}rgen Herzog and Rasoul~Ahangari Maleki, \emph{{K}oszul cycles and {G}olod
  rings}, Manuscripta Mathematica \textbf{157} (2018), no.~3, 483--495.

\bibitem{huneke2007ideals}
Craig Huneke, \emph{Ideals defining {G}orenstein rings are (almost) never
  products}, Proceedings of the American Mathematical Society (2007),
  2003--2005.

\bibitem{katthan2017non}
Lukas Katth{\"a}n, \emph{A non-{G}olod ring with a trivial product on its
  {K}oszul homology}, Journal of Algebra \textbf{479} (2017), 244--262.

\bibitem{peeva19960}
Irena Peeva, \emph{0-borel fixed ideals}, Journal of Algebra \textbf{184}
  (1996), no.~3, 945--984.

\bibitem{srinivasan1989algebra}
Hema Srinivasan, \emph{Algebra structures on some canonical resolutions},
  Journal of Algebra \textbf{122} (1989), no.~1, 150--187.

\bibitem{vandebogert2020dg}
Keller VandeBogert, \emph{{DG} structure on length 3 trimming complexes and
  applications to tor algebras}, arXiv preprint arXiv:2011.12324 (2020).

\bibitem{vandebogert2020trimming}
\bysame, \emph{Trimming complexes and applications to resolutions of
  determinantal facet ideals}, Communications in Algebra (2020), 1--20.

\end{thebibliography}
\addcontentsline{toc}{section}{Bibliography}

\end{document}